\documentclass{amsart}
\usepackage{ifthen}
\usepackage{enumerate}
\numberwithin{equation}{section}
\def\epsilon{\varepsilon}

\def\de{\delta}
\def\si{\sigma}
\def\ga{\gamma}

\def\ti{\tilde}

\def\ga{\gamma}

\def\beq{\begin{eqnarray}}
  \def\eeq{\end{eqnarray}}
\def\be{\beta}
\def\al{\alpha}
\def\ep{\epsilon}
\def\partt{\frac{\partial }{\partial t} }

\def\phi{\varphi}
\def\R{\mathbb R}
\def\inj{{\rm inj}}
\def\boundary{\partial}
\def\N{\mathbb N}

\def\part{\partial}

\def\curlR{\mathcal R}

\DeclareMathOperator{\vol}{vol}

\DeclareMathOperator{\Riem}{Riem}
\DeclareMathOperator{\Ric}{Ricci}
\DeclareMathOperator{\Ricci}{Ricci}

\DeclareMathOperator{\dist}{dist}

\def\counterword#1{%
  \ifthenelse{\ref{#1}=1}{one}{}%
  \ifthenelse{\ref{#1}=2}{two}{}%
  \ifthenelse{\ref{#1}=3}{three}{}%
  \ifthenelse{\ref{#1}=4}{four}{}%
  \ifthenelse{\ref{#1}=5}{five}{}%
  \ifthenelse{\ref{#1}=6}{sic}{}%
}

\newtheorem{theorem}{Theorem}[section]
\newtheorem{lemma}[theorem]{Lemma}

\theoremstyle{definition}

\theoremstyle{remark}
\newtheorem{remark}[theorem]{Remark}
\newtheorem{notation}[theorem]{Notation}

\begin{document}
\title[Ricci flow of regions with $\curlR \geq -1$ in dimension three
] {Ricci flow of regions with curvature bounded below in dimension
three}

\thanks{We would like to thank the referees for their comments, and Peter Topping for conversations
 on an earlier version of this paper. We hope that the
 changes we made as a result of these comments/conversations have made the paper more readable and easier to understand.}

%

\author{Miles Simon}
\address{Miles Simon:
Institut f\"ur Analysis und Numerik (IAN), Universit\"at Magdeburg, Universit\"atsplatz 2, 39106 Magdeburg, Germany}

\curraddr{}
\email{ msimon at ovgu point de}


\dedicatory{}


\begin{abstract}
We consider smooth complete solutions to Ricci flow with bounded
curvature on manifolds without boundary in
dimension three. Assuming an open
ball at time zero of radius one has sectional curvature bounded from below by -1, then
we prove estimates which show that compactly contained  subregions of
this ball will be smoothed out by the Ricci flow for a short but well
defined time interval. 
The estimates we obtain depend only on the initial volume of the ball and
the  distance from the compact region to the boundary of the initial
ball. Versions of these estimates for balls of radius r  follow using
scaling arguments.
\end{abstract}

\maketitle
\section{Introduction}
In this paper we consider 
smooth solutions $(M,g(t))_{t \in
[0,T) }$ to Ricci flow $$\partt g = -2 \Ricci(g)$$ as introduced and
first studied in R.Hamilton's paper \cite{HaThree}.
The solutions $(M,g(t))_{t \in
[0,T) }$ we consider are smooth  (in space and time),
connected, complete for all $t \in [0,T)$,
 and  $M$ has no  boundary.
We usually assume that the solution $(M,g(t))_{t\in[0,T)}$ {\it has
bounded curvature}, that is that 
 $\sup_{M \times [0,T)} |\Riem(x,t)|
< \infty$. The  value  $k_0 := \sup_{M\times[0,T)} |\Riem(x,t)| <
\infty$ will play no role in the estimates we obtain.

In the paper \cite{Per}, G. Perelman proved a Pseudolocality Theorem for
solutions of the type
described  above:  if a ball
${{}^0 B}_r(p_0) $ of radius $r>0$ in an $n$-dimensional manifold $(M^n,g(0))$ at time zero is {\it almost
  Euclidean} (see Section 10 in \cite{Per}), and $(M^n,g(t))_{t \in
[0,T) }$ is a complete solution to the Ricci flow with bounded curvature, then 
for small times $  t \in [0,\ep^2(n)r^2))$,  we have estimates on how the
curvature behaves on balls  $ {}^t B_{\ep(n) r}(p_0)$. There are a number
of versions of this theorem: see the introduction in the paper
\cite{SimSmoo} for references and further remarks.
In the paper \cite{SimSmoo} we generalised this result in the  two
dimensional setting. In particular we allow regions at time zero which
are not necessarily almost Euclidean: see Theorem 1.1 in \cite{SimSmoo} and the remarks before and after the
statement of  Theorem 1.1 there.
The purpose of this paper is to generalise this result to the three
dimensional setting. 
\begin{notation}
In this paper, $\curlR(g)$
always refers to curvature operator.
When we write $\curlR(g) \geq c$ for a constant $c \in \R$, then we mean that\\
$ \Riem(g)^{ik jl}\omega_{ik}\omega_{jl} \geq
cg^{ij}g^{kl}\omega_{ik}\omega_{jl}$ on $M$ for all two forms $\omega
= \omega_{ij}dx^i \otimes dx^j$, $\omega_{ij} = -\omega_{ji}$, where
$\Riem^{ijkl}$ is the full Riemannian curvature tensor.
A two form $\omega$ has {\it length one}, if $|\omega|_g^2:=
g^{ij}g^{kl}\omega_{ik}\omega_{jl} =1$.
\end{notation}
We show the following in this paper.

\begin{theorem}\label{threed}
Let $ r, v_0>0  $   and $0<\al<1$ be given.  
Let $(M^3,(g(t)) _{t \in [0,T)}$ be a smooth complete solution to
Ricci flow with bounded curvature and no boundary, and let $p_0 \in M$ be a point such that
\begin{itemize}
\item $\vol({{}^{0} B}_{r}(p_0)) \geq v_0r^3$ and
\item $\curlR(g(0)) \geq -  \frac {1 }{ r^2}$ on  ${}^{0}
  B_{r}(p_0)$.
\end{itemize} 
Then there exists an $N = N(v_0,\al) $ and  a $\ti v_0 = \ti v_0
(v_0) >0$ 
such that
\begin{itemize}
\item[(a)] $\vol({}^{t} B_{r}(p_0)) \geq  \ti v_0 r^3$
\item[(b)] $\curlR(g(t)) \geq -  \frac {N^2} {r^2} $ on  $ {}^{t}
  B_{r(1-\al)}(p_0)$
\item[(c)] $|\Riem| \leq \frac{N^2}{ t}$ on ${{}^tB}_{r(1-\al)}(p_0)$
\end{itemize}
as long as $t \leq \frac{r^2}{N^2}$  and $t \in [0,T)$.
\end{theorem}

\begin{remark} By scaling  it suffices to prove the theorem
  for $r=1$ 
\end{remark}
\begin{remark} The
regions which are considered are not
necessarily {\it almost Euclidean} at time zero (see the introduction in the paper
\cite{SimSmoo} for further
remarks and comments).
\end{remark}
\begin{remark} 
This localises the global results of Theorem 1.7   of \cite{SimThree} and
Theorem 1.9 of \cite{SimColl}  which proved a similar result for
the case that the curvature operator is bounded from below by minus
one on the whole manifold, and that the solution has bounded curvature and
$\vol({{}^0B}_1(x)) \geq v_0 >0$ for all $x$ in the manifold at time
zero.
\end{remark}


The above result (Theorem \ref{threed}) is obtained as a corollary of the following theorem
(Theorem \ref{threed2})
combined with  Theorem \ref{thm22new} (which is 
 a modified version of Theorem 2.2  of
\cite{SimSmoo}), as we explain in the last
section of this paper.
\begin{theorem}\label{threed2}
Let $ r, v_0>0  $ be given and $(M^3,g(t)) _{t \in [0,T)}$ be a smooth complete solution to
Ricci flow with bounded curvature  and no boundary. Let $p_0 \in M$ be a fixed  point and assume that
\begin{itemize}
\item $\vol({{}^{0} B_{s}(x)}) \geq  v_0 s^3$  for all $s>0$ and $x
  \in M^3$ which satisfy 
${}^{0} B_{s}(x) \subseteq {}^{0} B_{r}(p_0)$, and 
\item $\curlR(g(0)) \geq - \frac{1}{r^2}$ on  ${}^{0} B_{r}(p_0)$
\end{itemize}

Then there exists a (large) $K = K(v_0) $ and a (small) $\si_0 = \si_0(v_0) > 0$
such that
\begin{itemize}
\item[(i)]$\curlR(g(t))(x) (r-d_t(x,p_0) )^2 > -  K^2$
\end{itemize}
for all  $x \in \overline{{{}^{t} B}_{r- (K\sqrt{t}/\sqrt{\si_0} ) }(p_0)}$
and $ t \leq 
\frac{r^2\si_0}{ K^2} $ and $t \in [0,T)$.
Here $d_t(x,p_0) = d(g(t))(x,p_0)$ is the distance from $x$ to $p_0$
measured using $g(t)$.
\end{theorem}

\begin{remark}\label{proofvn}
We may change the result of the theorem to the statement 
'Then there exists a (large)  $N = N(v_0) $ and a (small) $\si_0 = \si_0(v_0) > 0$ 
such that
\begin{itemize}
\item[(i)]$\curlR(g(t))(x) (r-d_t(x,p_0) )^2 > -  \si_0 N^2$
\end{itemize}
for all  $x \in \overline{{{}^{t} B}_{r-N\sqrt{t}}(p_0)}$  which satisfy $ t \leq 
\frac{r^2}{N^2} $ and $t \in [0,T)$.'
by setting $N^2 = \frac{ K^2}{\si_0}$.
This is the statement that we shall prove.

\end{remark}

\section{Comments on the proof of Theorem \ref{threed}, and the 
use of Perelman's Pseudolocality Theorem therein}
The main ingredients of the proof of Theorem \ref{threed} are: (i)  Theorem 1.5 of \cite{SimSmoo}, 
(ii) The Pseudolocality Theorem of G. Perelman (Section 10 in
\cite{Per}) and  (iii) Theorem \ref{thm22new} of this paper (which is a
modified version of Theorem 2.2  of \cite{SimSmoo}).

The idea is essentially as follows.
Theorem \ref{threed} follows from Theorem \ref{threed2} and the results (slightly modified) of
\cite{SimSmoo}. 
So we have to prove  Theorem \ref{threed2}. We use the notation from
Remark \ref{proofvn}. We choose $N(v_0)>0$ large and $\si(v_0)$ small: they are specified in the proof.
The first part of the
proof of Theorem \ref{threed2} is a scaling argument which gets us into
a setting where the scaled  radius $r$ is now $L$, and the first valid time and point $t=t_0$ and $x=z_0$  where
$\curlR(g(t))(x) (L-d_t(x,p_0) )^2 > -  \si_0 N^2$ fails to hold
satisfies (after scaling): $t_0 \in [0,1]$, and the new radius 
$L$ is very large, in particular $L \geq N$, 
\begin{itemize}
\item[(a)] $\vol({{}^{0} B_{s}(x)}) \geq  v_0 s^3$  for all 
${}^{0} B_{s}(x) \subseteq {}^{0} B_{L}(p_0)$ and  
\item[(b)] $\curlR(g(t))(x) \dist_{L,t}^2(x) \geq -\si_0 N^2$
for all  $x \in \overline{{{}^{t} B}_{L-N\sqrt{t}}(p_0)}$  which satisfy $ t \leq t_0 \leq 1$, 
\item[(c)]  $\dist_{L,t_0}(z_0) = N$ and
$\curlR(g(t_0))(z_0)(\omega,\omega) = -\frac{\si_0 N^2}{
  \dist_{L,t_0}^2(z_0)} = -\si_0$ for a two form $\omega$ of length one
(w.r.t to $g(t_0)$),
\item[(d)] $\curlR(g(0))(x) \geq -\frac{1}{L^2} \geq -\frac{1}{N^2}$
  on  ${}^{0} B_{L}(p_0)$,
\end{itemize}
where $\dist_{L,t}(x) =  (L-d_t(x,p_0)) $ for $x \in {{}^tB}_L(p_0)$ and is
$0$ otherwise.\\
We next show that (a),(b),(c) and (d) lead to a contradiction if the
constant $\si(v_0)$ is small enough, and the constant $N(v_0)$ is
large enough. Here we give a {\bf rough} sketch of the proof idea. There are two cases that need to be considered, when obtaining this contradiction:
\\ Case i):  $t_0$ is {\it not too near to $1$  ($t_0 \leq  1-
  10\be_0$ with $\be_0 = \si_0^{1/4}$ suffices)}.\\
In this case, we see (see the proof), that $ {{}^{t}B}_{\be_0 N }(z_0)
\subseteq {{}^tB}_{L-N\sqrt{t} }(p_0)$ for all $t \leq t_0$, and hence
  we have the estimate $(b)$, at any point in the space-time cylinder
  of radius $\be_0 N$ centred at $z_0$ with base time $0$ and top time $t_0$,
$\cup_{t \in [0, t_0]}    {{}^{t}B}_{\be_0 N }(z_0) \times \{t\}$, and hence
$\curlR(\cdot,\cdot) \geq-1$ on the same space time cylinder of half
the radius, in view of (b). Note that, the radius $\be_0 N$ is very large by
assumption (see proof). 
Regularity estimates of previous papers (in particular the paper
\cite{SimSmoo}), which  do not rely on
 Perelman's Pseudolocality result, show us that the norm of the full
 curvature tensor is bounded by
$\frac{C(v_0) }{t}$ for some constant $C(v_0)$,  at any point $(x,t)$
in a  space time cylinder of a smaller (but still large enough) radius, with
the same centre point, and base resp. top time. Then, (d) and a
regularity result from \cite{SimSmoo}, tells us that $\curlR(z_0,t_0) \geq -\ep(N) \to 0$ as $N \to \infty$: this
contradicts $(c)$ if $N=N(v_0)$ is chosen large enough initially.
That is: the case $t_0 \leq 1- 10 \be_0$ follows from this scaling or 'blow up' argument and the  
(slightly modified) results of \cite{SimSmoo}, and the use of Perelman's
Pseudolocality Theorem is {\bf not} necessary in this case.\\
\\Case ii): The case that $ t_0 $ is near to one, that is $ 1\geq t_0\geq
1-10\be_0$.\\
This case  cannot be immediately  handled in the same way as Case i).
The reason is: it could be that we can't find a large enough radius
$R>0$ such that
${{}^tB}_{R}(z_0) \subseteq {{}^{t} B}_{L-N\sqrt{t}}(p_0)$ for all
$t\leq t_0$, and hence we
{\bf do not} have the estimate $(b)$ on some space-time cylinder
with large (enough) radius centred at $z_0$ with base time $0$ and top
time $t_0$.
In the extreme case we have $t_0 =1$, and hence  $d_1(p_0,z_0) = L-N$ (since
$\dist_{L,1}(z_0) = N$) and hence $ z_0$ is in the boundary of
${{}^{t_0} B}_{L-\sqrt{t_0}N}(p_0) ={{}^{1} B}_{L-N}(p_0)$ at time
$t_0$. To get around this problem we proceed as
follows. Using the method described in Case (i), we see that
$|\Riem(x,s)|\leq \frac{C(v_0)}{s}$ for all points $(x,s)$
on  some space time cylinder with large radius centred at  $z_0$ with
base time $0$ and top  time $t$, as long as 
 $t \leq 1-10\be_0$.
The (second) Pseudolocality Theorem of G. Perelman for times $t \in [1- 10
\be_0,t_0]$, combined with the estimates  which were obtained for $t
\leq 1- 10 \be_0 $, allows us to extend this estimate to
$|\Riem(\cdot,t)|\leq \frac{\ti C(v_0)}{t}$ for all $t \in [0,t_0]$ on
some space-time cylinder with large radius centred at  $z_0$ with
base time $0$ and top  time $t_0$.
Now we use the regularity result of \cite{SimSmoo}, as  in Case (i), and get that $\curlR(z_0,t_0) \geq -\ep(N) \to 0$ as $N \to \infty$, which
contradicts $(c)$ if $N=N(v_0)$ is chosen large enough initially.
\\
We write '{\bf rough} sketch' above, because many of the difficulties which
occur in the proof are avoided in this sketch. In particular, we
actually prove estimates on cylinders of the type explained above  with
arbitrary centre  points $y_0 \in {{}^0 B}_{L-N +
    \frac{3}{4} \be_0 N}(p_0)$ instead of $z_0$, and then we show, by proving estimates on
  how distances can change, that $z_0$ is in fact a  point in ${{}^0 B}_{L-N +
    \frac{3}{4} \be_0 N}(p_0)$, and hence the estimates of the type
  explained above (for $z_0$) do hold.\\
Note that the use of G. Perelman's Pseudolocality Theorem is very
necessary for this proof.
The difficult case is $t_0 \geq
1-10\be_0$. As we pointed out above, this corresponds to $z_0$ being
close to or in the boundary of
$B_{L-N \sqrt{t_0}}(p_0)$. In all such blow up arguments in geometric analysis, this is the difficult case and there is no
guarantee that this case can be dealt with. Whether this case can be
dealt with or not will depend on the flow being considered. For the  Ricci flow, the Pseudolocality Theorem enables us to
deal with this case.\\
Note that Theorem \ref{threed} (respectively Theorem \ref{threed2}) of
this paper is  {\it almost }  a truly local theorem: we only need
assumptions at time zero along with the assumption that the solution
we are considering has bounded curvature and is complete. 
We say {\it almost}, because we require an assumption on the solution
itself, namely that the solution has bounded
curvature and is complete, in order to apply the Pseudolocality Theorem.
The regularity theorem, Theorem 1.5, of the paper \cite{SimSmoo} is
not a truly local theorem: one
of the requirements of that theorem is that $|\Riem| \leq c/t$ for
some $c>0$ on
the ball ${{}^t B}_{r}(p_0)$  (for all $t \in [0,T]$) that we are considering.
This is a strong assumption on the solution.

\section{Some local results}\label{sec22new}
In this section we prove some lemmata, which follow readily from
previously proved results. These results will be required in the proof of Theorem \ref{threed2}.\\
First we prove a modified version of  Theorem 2.2 of \cite{SimSmoo}.
The result of the theorem below and that of Theorem 2.2 of \cite{SimSmoo}
differ in the following way.
In Theorem 2.2 of \cite{SimSmoo} condition (a) there was 
'$  [{\rm a}]: \vol({{}^{t}B}_{r}(p_0))\geq  v_0 r^n $ for all $t
\in [0,T)$'. Here we only require $\vol({{}^{0}B}_{r}(p_0))\geq v_0
r^n $ at time zero.

\begin{theorem}\label{thm22new}
Let $  r,V,v_0>0,1>\al >1/2 $ and $(M^n,g(t))_{t \in [0, T)}$  be a
smooth, complete solution to Ricci-flow with no boundary
which satisfies
\begin{itemize}
\item[{\rm (a)}] $\vol({{}^{0}B}_{r}(p_0))\geq  v_0 r^n $, 
\item[{\rm (b)}] $ \curlR (x,t) \geq -\frac{V}{r^2}$ for all $t  \in [0,T) , x \in  {}^t B_{r}(p_0) $.
\end{itemize}
Then, there exist  $ 0< m_0 = m_0(n,v_0,\al,V),$  $c_0= c_0(n, v_0,\al,V)< \infty$ and
$\ti v_0=\ti v_0(n, v_0,V) >0$ such that
\begin{itemize}
 \item[{\rm (c)}]
$|\Riem(x,t)| <\frac{c_0}{ t}$ for all  
$x \in  {}^t B_{r (1-\al)}(p_0),$ $t  \in [0, m_0 r^2) \cap
[0,T),$ and \item[{\rm (d)}]
$\vol({{}^{t}B}_{s}(p_0))> \ti v_0 s^n$ for all $t  \in  [0,m_0 r^2
)\cap[0,T),$ for all $s<r$.
\end{itemize}
\end{theorem}
\begin{remark}
Note that here we do not require that $(M^n,g(t))_{t\in[0,T)}$ is a
solution with bounded curvature.
\end{remark}
\begin{proof}
Let $(M,g(t))_{t\in [0,T)}$ be as in the statement of the theorem.
Without loss of generality, after scaling, we have $r=1$.  We prove
the case  $s=1$ : the general statement in $(d)$ then follows
from the Bishop-Gromov comparison principle.\\
We know that $\vol( {{}^0  B}_{1/800}(p_0)) \geq
V_0(v_0,V,n) >0$ due to the Bishop-Gromov volume comparison principle. 
From the Appendix, Theorem \ref{appendixthm},  we see that the
following is true: there exists an $\ep_0 =
\ep_0(V_0,n) = \ep_0(v_0,V,n) >0$ such that if $ d_{GH}( {{}^t
  B}_{1/800}(p_0),  {{}^0  B}_{1/800}(p_0) ) \leq \ep_0$ for some $t
\in [0,T)$, then $\vol ({{}^t
  B}_{1/800}(p_0)) > \ep_0$.
Assume there is a first time $S \in (0,T)$ where  $\vol ({{}^t
  B}_{1/800}(p_0)) >\ep_0$ is violated: 
$\vol ({{}^t
  B}_{1/800}(p_0)) >\ep_0$ for all  $0\leq t <S $ and $\vol ({{}^{S}
  B}_{1/800}(p_0)) = \ep_0$.
From Theorem 2.2 of \cite{SimSmoo} we have:
$|\Riem|<\frac{N^2}{t}$ on ${{}^tB}_{1-\al}(p_0)$ for all $t \leq
\min( \hat T(\ep_0,n,\al,V),S)  = \min( \hat T(v_0,V,n,\al) ,S)$ for
some
$N = N( \ep_0,n,\al,V) = N(v_0,V,n,\al) $ , and $\hat T = \hat T(\ep_0,n,\al,V)  = \hat T(v_0,V,n,\al) >0$.
But then, this estimate, (b) and \cite{HaForm} (Lemma 17.3 combined
with Theorem 17.4)  imply (for such $t$) that 
$e^{ta(n) V}d_0(x,y) \geq d_t(x,y) \geq d_0(x,y)- a(n) N \sqrt t $
for all $x,y \in {{}^tB}_{1/200}(p_0)$ (see section 3 of \cite{HaFour}), since any
geodesic at time $t$ between such $x$ and $y$ must lie in 
${{}^tB}_{1/2}(p_0)$, due to the triangle inequality.
This means that
${{}^t B}_{1/700}(p_0) \subseteq {{}^0 B}_{1/400}(p_0) \subseteq {{}^tB}_{1/200}(p_0)$
and 
\begin{eqnarray}
(1+ \ep_0^2) d_0(x,y) \geq d_t(x,y) \geq d_0(x,y) -\ep_0^2 \mbox{ on }\label{GHep}
{{}^t B}_{1/700}(p_0)
\end{eqnarray}
for all such $t$ which also satisfy $t \leq \ti
T(v_0,V,\al,n)$, where $\ti T(v_0,V,\al,n)>0$ is small enough.
Assume $S \leq \min(\ti T(v_0,V,\al,n), \hat T(v_0,V,n,\al) ,T)$. Then we  have
$d_{GH}({}^{S}  B_{1/800}(p_0),  {}^{0}  B_{1/800}(p_0)) <
\ep_0$  (a Gromov-Hausdorff
approximation \\$f: {}^{S}  B_{1/800}(p_0) \to {}^{0}
B_{1/800}(p_0)$ is given by $f(x) = x$ for $x \in {}^{S}
B_{1/800}(p_0) \cap {}^{0} B_{1/800}(p_0)$ and $f(x) =\ti x$ for $x
\in  {}^{S}
B_{1/800}(p_0) \backslash {}^{0} B_{1/800}(p_0)$
where $\ti x \in  {}^{0} B_{1/800}(x_0)$ is an arbitrary point with
$d_0(x,\ti x) \leq 10 \ep^2_0$: such a point exists in view of the
inequalities \eqref{GHep}).
This is a contradiction to the definition of $\ep_0$.
\end{proof}

The next lemma is an integrated version of  Lemma 8.3 (b) of Perelman, \cite{Per}, in the case that the curvature behaves like a constant divided by time. 
\begin{lemma}\label{Littlericci}
For any $j_0,\ell>0, n\in \N$, there exists a constant $a(n,j_0,\ell)$ such that the following is true.
Let $(M^n,g(t))_{t\in [0,T]}$ be a  complete smooth solution to Ricci
flow with bounded curvature, and no boundary, and let $s_0 \leq \min(1,T)$. 
Assume that $y_0,x_0 \in M$ and that 
$|\Ricci(\cdot,t)| \leq \frac{\ell}{t}$ on both ${^{t}B}_{j_0}(x_0)$ and
${^{t}B}_{j_0}(y_0)$ for all $ t \in [0,s_0]$.
Then $d_s(y_0,x_0)
\geq d_0(y_0,x_0) -a(n,j_0,\ell)$ for all $0 \leq s \leq s_0$.
\end{lemma}
\begin{proof}

For $t \leq j^2_0$, we have $ |\Ric(x,t)| \leq \frac{\ell}{t}$ 
for any $x \in {{}^t B}_{\sqrt t} (x_0)$ and for any $y \in {{}^t B}_{\sqrt t} (y_0),$
since $\sqrt{t} \leq j_0$. Hence we may apply Lemma  8.3 (b) of \cite{Per} to this
with $t_0,K,r_0$ of Lemma  8.3 (b) of \cite{Per} given by  $t_0 =t$, $K = \frac{\ell}{t}$, $r_0 =\sqrt{t}$ to obtain:
\begin{eqnarray}
\partt d_t(x_0,y_0) && \geq -2(n-1)(\frac{2}{3}Kr_0 + r_0^{-1}) \cr
&& = -2(n-1)(\frac{2\ell\sqrt{t}}{3 t} + \frac{1}{\sqrt{t}} )\cr
&& =: \frac{1}{\sqrt{t} }( -\frac{4(n-1)\ell}{3} + 1),\label{integrate1}
\end{eqnarray}
for $t \leq j_0^2$,
where the time derivative is to be understood in the sense of forward
difference quotients.
For $t \in [j_0^2,s_0]$, we have that 
 $|\Ricci(\cdot,t)| \leq \frac{\ell}{j_0^2}$ and hence
applying  Lemma  8.3 (b) of \cite{Per} with  $t_0 =t$, $K = \frac{\ell}{j_0^2}$, $r_0 = j_0$ we get
\begin{eqnarray}
\partt d_t(x_0,y_0) && \geq -2(n-1)(\frac{2}{3}Kr_0 + r_0^{-1}) \cr
&& = -2(n-1)(\frac{2  j_0 \ell }{3 j_0^2} + \frac{1}{j_0} ) \label{integrate2}
\end{eqnarray}
for $t \in [j_0^2,s_0]$,  
where the time derivative is to be understood in the sense of forward
difference quotients.
Integrating first Equation \eqref{integrate1} from $0$ to $j_0^2$ and then Equation 
\eqref{integrate2} from $j_0^2$ to $s$ gives us the result (if $s \leq j_0^2$ then we
merely integrate Equation \eqref{integrate1} from $0$ to $s$).
\end{proof}

The last lemma of this section is a  technical lemma, which uses some facts from differential geometry. 
\begin{lemma}\label{littletechnical}
For every  $\ti v_0, \ell >0$ and $\de \in (0,1)$ the following is true.
Let $(M^n,g)$ be a smooth Riemannian manifold, $x_0 \in
M$, with no boundary, such that the closure of
${B}_1(x_0)$ is compactly contained in $M$ and 
\begin{itemize}
\item[(i)] 
$|\Riem| \leq \ell $ on ${B}_1(x_0)$
 and 
\item[(ii)]
$\vol({B}_1(x_0)) \geq \ti v_0$.
\end{itemize}
 Then there exists an $R_0(n,\ell,\ti v_0,\de) > 0$ such that
\begin{eqnarray*}
|\Riem|(\cdot) \leq \frac{1}{R_0^2} \mbox { on } { B}_{R_0}(x_0)\\
\vol({B}_{R_0}(x_0)) \geq \omega_n (1-\de) R_0^n,
\end{eqnarray*}
where $\omega_n$ is the volume of an $n$-dimensional Euclidean ball of radius one.
\end{lemma}

\begin{proof}

The inequalities (i) and  (ii) imply  that\\
$\inj(g)(x_0) \geq i_0(n,\ti v_0,\ell) > 0 $ for
some $i_0(n,\ti v_0,\ell) >0$ in view of the estimate of
J. Cheeger/M. Gromov/M. Taylor, (4.22) in  Theorem 4.3 of \cite{CGT}. Hence, using Riemannian normal coordinates (see Theorem 1.53 and the
proof thereof in \cite{Aub}), we see that
\begin{eqnarray*}
\vol( { B}_{ r}(x_0) ) \geq \omega_n (1-\de) r^n
\end{eqnarray*}
 for all $r \leq R_0(n,\ell,i_0(n,\ti v_0,\ell), \de) = R_0(n,\ell,\ti v_0,\de)$, 
if $R_0(n,\ell, \ti v_0,\de)>0$ is small enough,
where  $\omega_n$ is the volume of the unit  ball in
$n$-dimensional Euclidean  space.
Without loss of generality, we also have 
\begin{eqnarray*}
|\Riem(y)| \leq \frac{1}{R^2_0}
\end{eqnarray*}
on ${B}_{R_0}(x_0)$,  since without loss of
generality $\frac{1}{R^2_0(n,\ell,\ti v_0,\de)} \geq \ell$: if not, decrease
$R_0(n,\ell,\ti v_0,\de)$ until it is.

\end{proof}

\section{Proof of Theorem \ref{threed2}}

In order to obtain local estimates we introduce the following distance
function for balls which are evolving in time under the Ricci
flow. Let $(M,g(t))_{t \in [0,T)}$ be a solution to Ricci flow.
Let ${{}^t B}_r(p_0)$ be the open ball of radius $r$ at time $t$ centred at $p_0
\in M$. Notice that $x \in {{}^t B}_r(p_0)$ does not necessarily
guarantee that $x \in {{}^s B}_r(p_0)$ for a different time $s$. For $x \in  {{}^t B}_r(p_0)$ we define 
\begin{eqnarray*}
\dist_{r,t}(x):= (r-d_t(p_0,x))  
\end{eqnarray*}
 where $d_t(p_0,x)$ is the distance from
$x$ to $p_0$ measured using the evolving metric $g(t)$.
Cut-off functions of this  type were used in the papers
\cite{Chen},\cite{SimLoc} and \cite{SimHab} in combination with Ricci flow to
prove that local estimates hold, {\bf if} one a priori assumes that
the curvature satisfies a bound of the type $|\Riem(\cdot,t)| \leq
c/t$.

Notice that $0 < \dist_{r,t}(x) \leq r $ for all $x \in  {{}^t B}_r(p_0)$.
$\dist_{r,t}(x)$ is a measure of how far the point $x$ at time $t$ is from the boundary of
${{}^tB}_r(p_0)$.
In the case that $g(t) = \de$ the Euclidean metric on $\R^n$, then we
see that $\dist_{r,t}(x):= (r-d_t(p_0,x))= (r- |x-p_0|)$ is the distance
from $x$ to the boundary of $B_r(p_0)$.
Due to scaling it will be sufficient to consider the case $r=1$.
Let ${{}^{0}B}_1(p_0)$ be a ball at time zero with curvature bounded
from below by minus one.
The following theorem implies a lower bound on the curvature at $x \in
{{}^{t}B}_1(p_0)$ depending on $\dist_{1,t}(x)$ at later times for a well defined time interval, as long
as $\dist_{1,t}^2(x) \geq N^2t$ where $N^2$= $N^2(v_0)$  is sufficiently
large, and $v_0$ is a lower bound (at time zero) on the volume quotient of balls
contained in the ball we are considering, and the curvature of ${{}^{0}B}_1(p_0)$ at time
zero is bounded from below by $-1$.
Combining this theorem with the results  of Section \ref{sec22new} will imply
the result of Theorem \ref{threed}  stated in the introduction (see
Section \ref{secthreed} for the proof of Theorem \ref{threed}). Here
we restate Theorem \ref{threed2} for the case $r=1$ using the notation
that we just introduced, and Remark  \ref{proofvn}.

\begin{theorem}\label{threedthm}
Let $(M^3,(g(t)) _{t \in [0,T)}$ be a smooth complete solution to
Ricci flow with bounded curvature  and no boundary and
$v_0 >0$. Let $p_0 \in M$ be a point
such that
\begin{itemize}
\item $\vol({{}^{0} B_{r}(x)} )\geq  v_0 r^3$  for all $x \in M^3$,
  and $r>0$ which satisfy
${}^{0} B_{r}(x) \subseteq {}^{0} B_{1}(p_0)$,
and
\item $\curlR(g(0)) \geq - 1$ on  ${}^{0} B_{1}(p_0)$
\end{itemize} 
Then there exists an $N = N(v_0) $, $\si_0 = \si_0(v_0) > 0$,
such that
\begin{itemize}
\item[(i)] $\curlR(g(t))(x) \dist_{1,t}^2(x) > - \si_0 N^2$
\end{itemize}
for all  $x \in \overline{ {{}^{t} B}_{1-N\sqrt{t}}(p_0)} $  which satisfy 
$ t \leq 
\frac{1}{N^2} $ and $t \in [0,T)$.
\end{theorem}

\begin{proof}
$v_0$ is fixed throughout  the proof and $\si_0 = \si_0(v_0)>0$ is a small
constant determined in the proof.

Assuming the theorem is false for some given $N = N(v_0,\si(v_0)) = N(v_0)$ large  and $\si_0 =
\si_0(v_0)>0$ small (to be determined in the proof), then there must be a first time $t_0 \leq 
\frac{1}{N^2}<T$
where the theorem fails. That is (i) is violated at $t_0$.
We show that if $\si(v_0)>0$ is chosen small enough, and  $N=N(v_0)>0$
is chosen large enough, that this leads to a contradiction. Let $\be_0
= \si_0^{1/4} $ throughout the proof. 
At the end of the theorem, see Remark \ref{constants}, we give a
precise explanation of how $N$ and $\si$ can be chosen at this point of the theorem.
\\(i) is
violated at some first time $t_0$ means
that we can find a $z_0 \in
{{}^{t_0}B}_{1}(p_0)$ and $0<t_0  \leq 
\frac{1}{N^2}<T$ with $ \dist_{1,t_0}^2(z_0) \geq N^2{t_0}$
such that
$\curlR(g(t_0))(z_0)(\psi,\psi) \dist_{1,t_0}^2(z_0) = -\si_0 N^2$ for
some two form $\psi$ of length one (measured with respect to $g(t_0)$), and the conclusions of the theorem are
correct for $0<t < t_0$.
Let $L^2= \frac{N^2}{\dist_{1,t_0}(z_0)^2}$. Remembering that
$\dist_{1,t_0}(z_0)^2 \leq 1$ we see that  $L \geq N$. 
We scale our solution by an appropriate constant, so that the new
solution has 
$\ti \dist_{L,\ti t_0}^2(z_0) = N^2$ at the  new time $\ti t_0$ which corresponds
to the old time $t_0$ in the original solution: 
define $\ti g(\cdot, \ti t):= L^2 g(\cdot,\frac{\ti t}{L^2})$. 
This solution is defined for $\ti t \in [0,\ti T = L^2T\geq N^2T)$.
Then define for $x \in {{}^{\ti t}B}_L(p_0)$
\begin{eqnarray}
 {\ti \dist}_{L,\ti t}^2(x) &&:= (L- \ti d_{\ti   t}(x,p_0))^2 \label{newzero}\\
&&=   L^2(1-  d_{t}(x,p_0))^2\cr
&&= L^2 \dist_{1,t}^2(x) 
\end{eqnarray}
where $t = \frac{\ti t}{L^2} $ and $ \ti d_{\ti t}(x,p_0)$ is the distance measured
with respect to $\ti g(\ti t)$.

This value is positive since $x \in {{}^{\ti t}B}_L(p_0)$ if and only if $ \ti
d_{\ti t}(x,p_0) < L$.
Using the definition of $\ti t$ and $\ti{\dist}_{L,\ti t}$ we see that 
\begin{eqnarray}
 \dist_{1,t}^2(x) \geq N^2t \iff  \ti{\dist}_{L,\ti t}^2(x) \geq N^2{\ti
   t}. \label{newone}
\end{eqnarray}
Also,
\begin{eqnarray}
 \ti{\dist}_{L,\ti t_0}^2(z_0)  =
L^2 \dist_{1,t_0}^2(z_0) 
=  \frac{N^2}{\dist_{1,t_0}(z_0)^2}\dist_{1,t_0}^2(z_0) 
= N^2\label{newtwo}
\end{eqnarray} and
 \begin{eqnarray*}
\ti t_0  = t_0 L^2 = t_0 \frac{ N^2}{\dist_{1,t_0}(z_0)^2} \leq
t_0\frac{N^2}{N^2t_0} = 1. 
\end{eqnarray*}
Notice that 
\begin{eqnarray}
\curlR(\ti g(\ti t))(x)\ti\dist_{L,\ti t}^2(x) = \curlR(g(t))(x) 
\dist_{1,t}^2(x),\label{newthree}
\end{eqnarray}
in view of the definition of $\ti \dist_{L,\ti t}$ and the way curvature changes
under scaling.

For ease of reading we will denote the solution   $\ti g(x,\ti t)$ by
$g(x,t)$.  Also $\ti t_0$ will be denoted by $t_0$,  $\ti t$ by $t$,
and  $\ti{\dist}_{L,\ti
  t}$ by $\dist_t$ ($L$ is now fixed).
Then we now have 
\begin{itemize}
\item[(a)] $\vol({{}^{0} B_{s}(x)}) \geq  v_0 s^3$  for all 
${}^{0} B_{s}(x) \subseteq {}^{0} B_{L}(p_0)$ and  
\item[(b)] $\curlR(g(t))(x) \dist_t^2(x) \geq -\si_0 N^2$
for all  $x \in \overline{ {{}^{t} B}_{L-N\sqrt{t}}(p_0)}$  which satisfy 
$ t \leq t_0 \leq 1$, 
\item[(c)] 
$\curlR(g(t_0))(z_0)(\omega, \omega) = -\frac{\si_0 N^2}{
  \dist_{t_0}^2(z_0)} = -\si_0$  for the two form $\omega = L^2
\psi$ which has length one with respect to (the new) $g(t_0)$,
\item[(d)] $\curlR(g(0))(x) \geq -\frac{1}{L^2} \geq -\frac{1}{N^2}$
  on  ${}^{0} B_{L}(p_0)$.
\end{itemize}
The first two inequalities are scale invariant (if they hold for some
solution, then they hold for a
scaling of the  Ricci flow after adjusting the delimiters, assuming
that we have defined the new $\dist_t$ for the scaled solution as in
\eqref{newzero}: cf. \eqref{newone} and \eqref{newthree}). In the  third
equality we used the fact that (after scaling) ${\dist}_{
  t_0}^2(z_0)  = N^2$ (see \eqref{newtwo} ): after scaling, we also have
$\curlR(g(t_0))(z_0)( \omega,  \omega) \dist^2_{t_0}(z_0) = -\si_0
N^2$
 and hence $\curlR(g(t_0))(z_0) ( \omega, \omega) = -\frac{\si_0
  N^2}{\dist^2_{t_0}(z_0)} = -\si_0$.
The last inequality, (d),  follows since we are scaling by $L^2$ and we showed
$L \geq N$. For all $0\leq t\leq t_0$we have
\begin{eqnarray*}
\{ x \in \overline{{}^{t} B_{L}(p_0)} \  |  \ \dist_t^2(x) \geq N^2t
\} = 
\overline{{}^tB_{L-N\sqrt{t}}(p_0)}. 
\end{eqnarray*}
Let $x_0 \in \overline{ {{}^0 B}_{ L -N + N\be_0}(p_0)}$ be an arbitrary
point.
Clearly ${}^0 B_{\be_0 N}(x_0) \subseteq  \overline{{}^0
  B_{L}(p_0)}$, in view of the triangle inequality
 (we are using that $\be_0 \leq 1/2$, which we always assume).
\\
The rest of the proof is broken up into three steps.
\\
{\bf Step 1 } For an arbitrary $x_0 \in \overline{{{}^0 B}_{ L -N(1
   -\be_0)}(p_0)} = \overline{{{}^0 B}_{ L -N +
   N\be_0}(p_0)} $ we show that ${}^t B_{\be_0 N}(x_0) \subseteq
\overline{  {{}^{t} B}_{L-N\sqrt{t}}(p_0) } 
$ for all $t \leq 1- 10 \be_0 $  as long as $t\leq
t_0$,  and $N=N(v_0) >0 $ is sufficiently large. Using the estimates of Theorem \ref{thm22new}, we will then see that this
guarantees that $|\Riem(\cdot,t)|\leq \frac{c_0(v_0)}{t}$ on  ${{}^t
  B}_{\frac{1}{\be_0}}(x_0)$, and that $\vol({{}^t B}_1(x_0)) \geq \ti
v _0(v_0)>0$  for all $t \leq  \min(1 - 10 \be_0,t_0)$ , for some
constants $c_0(v_0),\ti
v_0(v_0) >0$.
\hfill\break
{\bf Now we present the details of Step 1}.\\
Let $x_0 \in \overline{{{}^0 B}_{ L -N + N\be_0}(p_0)}$ be arbitrary.
 We know that  $x_0  \in
\overline{  {{}^{t} B}_{L-N\sqrt{t}}(p_0) } $ is valid, if and only if
\begin{eqnarray}
d_t(x_0,p_0) \leq L - N\sqrt{t}. \label{disty}
\end{eqnarray}
In the following we only consider $t$ such that $0\leq t \leq t_0$,
where $t_0\leq 1$ was defined at the beginning of the proof.
Hence, starting at time zero and going forward to time $t$, as long as $x_0 \in
\overline{  {{}^{t} B}_{L-N\sqrt{t}}(p_0) } $ remains valid, any length minimising geodesic (with respect to the
metric at time $t$) from
$p_0$ to $x_0$ must also completely lie  in
$\overline{{}^tB_{L-N\sqrt{t}}(p_0)} $. 
At all points $y$ on such a geodesic we have 
\begin{eqnarray*}
\curlR(g(t))(y) \dist_t^2(y) \geq -\si_0 N^2, 
\end{eqnarray*}
in view of $(b)$.
Using $\dist_t(y) = L- d_t(y,p_0)$,  we see that this means
\begin{eqnarray*}
\curlR(g(t))(y)  && \geq -\frac{\si_0 N^2}{\dist^2_t(y)} \cr
 && = -\frac{\si_0 N^2}{(L - d_t(y,p_0))^2}
\end{eqnarray*}
for such $y$.

Using this inequality in the
evolution equation for the distance (Lemma 17.3 of \cite{HaFour})  we
see (as long as $x_0 \in \overline{  {{}^{t} B}_{L-N\sqrt{t}}(p_0) } $ remains valid)
\begin{eqnarray*}
\partt d_t(x_0,p_0) && \leq \sup_{\ga \in X_t} \int_0^{d_t(x_0,p_0)}
-2\Ricci(\ga(s),t) ds \cr
&& \leq \sup_{\ga \in X_t} \int_0^{d_t(x_0,p_0)} \frac{\si_0 20 N^2}{(L-s)^2}ds\cr
&&= \frac{20 \si_0 N^2}{L-s}|_{s=0}^{s = d_t(x_0,p_0)}\cr
&& = \frac{ 20 \si_0 N^2}{L- d_t(x_0,p_0)}  -   \frac{20 \si_0 N^2}{L}\cr
&&\leq \frac{20 \si_0 N^2}{L- d_t(x_0,p_0)} \cr
&& \leq \frac{ 20 \si_0 N^2}{N\sqrt{t}} \cr
&& =\frac{20 \si_0 N}{ \sqrt{t}} \cr
&& = \frac{20\be_0^4 N}{\sqrt t}
\end{eqnarray*}
where $X_t$ is the set of distance minimising geodesics from $p_0$ to
$x_0$ at time $t$ (that is, measured with respect to the metric $g(t)$
at time $t$)
parameterised by arclength, and we have used inequality \eqref{disty}.
Here,  $\partt$ is to be understood in the sense of forward difference
quotients: see chapter 17 of \cite{HaFour} for more details.
Integrating in time from $0$ to $t$, we see that this means 
\begin{eqnarray}\label{distyy}
d_t(x_0,p_0) && \leq d_0(x_0,p_0) +  40\be_0^4 Nt^{1/2} \cr
&& \leq d_0(x_0,p_0) +  \be_0^2 N \cr
 && \leq (L - N + N\be_0) + \be_0^2N \cr  
&& \leq L- (1-2\be_0)N 
\end{eqnarray}
 for all $t \leq t_0 (\leq 1)$ as long as  $x_0 \in
\overline{  {{}^{t} B}_{L-N\sqrt{t}}(p_0) } $  remains true, where we
have used  that
 $x_0 \in \overline{{{}^0 B}_{ L -N + N\be_0}(p_0)}$ (and $\be^2_0
 \leq \frac{1}{40}$ which we will always assume). Restrict now only to $t\leq 1-10\be_0$
in the above argument.
Using the fact that $-(1-2\be_0) \leq -\sqrt{t} -\be_0$  for such
times \footnote{Note that $-(1-2\be_0) \leq -\sqrt{t}- \be_0$ if and only if
$ \sqrt t \leq 1- 3 \be_0$ if and only if 
$ t \leq (1-3\be_0)^2 = 1 -6\be_0 + 9 \be^2_0$ and hence
 $t \leq 1- 10\be_0 $ implies $t \leq 1- 6\be_0 + 9\be^2_0$ implies 
$-(1-2\be_0) \leq -\sqrt{t}-\be_0$ 
as claimed},
and  inequality  \eqref{distyy} , we see that
\begin{eqnarray}
d_t(x_0,p_0)  && \leq L- (1-2\be_0)N \cr
&& \leq L- N \sqrt t  - \be_0 N
\label{useful}
\end{eqnarray}
 for all $t \leq \min(t_0,1-10\be_0)$ as long as  $x_0 \in
\overline{  {{}^{t} B}_{L-N\sqrt{t}}(p_0) } $  remains true, and hence
$x_0 \in \overline{  {{}^{t} B}_{L-N\sqrt{t}}(p_0) } $ will not be violated as long as $t \leq
\min(t_0,1-10\be_0)$.
Furthermore, the triangle inequality combined with \eqref{useful}
implies
 $${}^t B_{\be_0 N}(x_0) \subseteq
\overline{  {{}^{t} B}_{L-N\sqrt{t}}(p_0) } $$ will not be violated 
as long as $t \leq \min(1-10\be_0,t_0)$:  $y \in {}^t B_{\be_0 N}(x_0)$
implies
\begin{eqnarray*}
 d_t(y,p_0) && \leq d_t(y,x_0) + d_t(x_0,p_0) \leq \be_0 N + L- N
\sqrt{t} - \be_0 N  \\
&& = L- N \sqrt{t} 
\end{eqnarray*}
for such $t$, in view of the inequality \eqref{useful}

The lower bound on the curvature, (b), is therefore
valid on  ${}^t B_{\be_0 N}(x_0)$ as long as $t \leq 1-10\be_0$ and $
t\leq t_0$, and
hence, for $x$ in the ball of half the radius, $x \in {{}^t B}_{\frac{\be_0 N}{2}}(x_0)$, we have
\begin{eqnarray}
\curlR(g(t))(x) &&\geq
-\frac{\si_0 N^2}{\dist_t^2(x)}\cr
&&\geq  -\frac{\si_0 N^2}{ (N\sqrt{t} +\frac{N\be_0}{2})^2}\cr
&&\geq - \frac{4\si_0}{\be^2_0}\cr
&& = -4\be_0^2  \ (\geq -1)
\label{lowerbound}
\end{eqnarray}
($\si_0>0$ was chosen  to be  $ \si_0 = \be^4_0 $) for all $t \leq
1-10\be_0,$ $t \leq t_0,$ in view of the fact that 
\begin{eqnarray*}
\dist_t(x) && = (L - d_t(x,p_0)) \cr
&&  \geq (L- d_t(x,x_0) - d_t(x_0,p_0)) \cr 
&& \geq ( L - \frac{\be_0 N}{2} -d_t(x_0,p_0) ) \cr 
&& \geq ( L - \frac{\be_0 N} {2} -L + N \sqrt{t} +  \be_0 N ) \cr
&& =  (\frac{\be_0 N}{2} + N \sqrt{t} )
\end{eqnarray*} for $x \in {{}^t B}_{\frac{\be_0 N}{2}}(x_0)$, which
follows from  the definition of
$\dist_t(x)$, the triangle inequality and inequality \eqref{useful}. 
Choosing $V=16$, $\al = 1/2$, $r = \frac{2}{\be_0}$ in Theorem
\ref{thm22new} (this gives us $-\frac{V}{r^2} = -4 (\be^2_0)$), we
see that
\begin{eqnarray*}
&&|\Riem(\cdot,t)| \leq \frac{c_0(v_0)}{t}   \mbox  { on } {}^t
B_{\frac{1}{\be_0}}(x_0)  , \  \mbox{ and } \  \cr
&& \vol( {{}^t B}_s(x_0)) \geq \ti v_0(v_0)s^3 , \ \mbox{ for all } \
s\leq \frac{1}{\beta_0}, \cr
&& \ \ \ \ \mbox{ for all } \ t \leq \min(1-10\be_0,t_0,
\frac{m(v_0)}{\be_0^2} ), 
\end{eqnarray*}
since $N$ is large enough: we are assuming that 
 $\frac{N \be_0}{2} \geq \frac{2}{
   \be_0}$, and so 
${{}^t B}_{\frac{2}{\be_0}}(x_0)   \subseteq {{}^t B}_{\frac{\be_0 N}{2}}(x_0)  $ and so the conditions of 
Theorem \ref{thm22new} are satisfied in view of \eqref{lowerbound} and
(a). Note that the dependancy of the constants $c_0,m_0,\ti v_0$ from 
Theorem \ref{thm22new} is  $c_0 =c_0(n,v_0,\al,V) = 
c_0(3,v_0,1/2,16) = c_0(v_0)$,  $m_0 = m_0(n,v_0,\al,V)=
m_0(v_0)>0$, and $\ti v_0=\ti v_0(n, v_0,V) = \ti v_0(v_0)>0$ 
and $c_0,m_0, \ti v_0$ {\bf do not} depend on $N$ or $\si_0$:
decreasing $\si_0$ or increasing $N$ will not affect
$c_0(v_0),m_0(v_0)$ or $\ti v_0(v_0)$. We assume that
$\be_0^2 = \si_0^{1/2}  \leq m_0(v_0)$, so that 
\begin{eqnarray}\label{curvest}
&&|\Riem(\cdot,t)| \leq \frac{c_0(v_0)}{t}   \mbox  { on } {}^t
B_{\frac{1}{\be_0}}(x_0)  , \  \mbox{ and } \  \cr
&& \vol( {{}^t B}_s(x_0)) \geq \ti v_0(v_0)s^3 , \ \mbox{ for all } \
s\leq \frac{1}{\beta_0}, \cr
&& \ \ \ \ \mbox{ for all } \ t \leq \min(1-10\be_0,t_0)
\end{eqnarray}
in view of the fact that $t_0 \leq 1$.
Let $\ep(3), \de(3)$ be the constants appearing in the second
Pseudolocality Theorem of G. Perelman, Theorem 10.3 in \cite{Per}, in
the case that $n=3$ (as it is here).
From Lemma \ref{littletechnical}
  with $n=3$, $\ell= 2c_0(v_0)$, $\ti v_0 =\ti v_0(v_0),$ $\de =  \de(3),$ and $T_0 = \min(
  1-10\be_0,t_0),$ 
we see that there exists an $R_0 =R_0(3,\ell,\ti v_0,\de) =
R_0(3,2c_0(v_0),\ti v_0(v_0), \de(3)) = R_0(v_0) >0$ such that
\begin{eqnarray}
|\Riem|(\cdot,t) \leq \frac{1}{R_0^2} \mbox { on } {{}^{t} B}_{R_0}(x_0)\label{curvestshi2}\\
\vol({^{t} B}_{R_0}(x_0)) \geq \omega_3 (1-\de) R_0^3\label{localinfo2}
\end{eqnarray}
for all $ \frac{1}{2} \leq t \leq \min( 1-10\be_0,t_0)$.\\
It is helpful to notice the following at this stage: at the moment we have the freedom to choose $\be_0= \si_0^{1/4}>0$ as
small as we
like. Decreasing $\si_0$ (and hence $\be_0$) or increasing $N$ will not change the
constant $c_0(v_0)$ we obtained above, and hence will not change $R_0(v_0) = R_0(3,c_0(v_0),v_0,\de(3))  $  we obtained above.
\hfill\break
{\bf This finishes Step 1}.
\\
{\bf Step 2}\\
In Step 2 we use the estimates from Step 1 and the (second) 
Pseudolocality result of G. Perelman to show that 
 $|\Riem(\cdot,t)|\leq \frac{\ti c(v_0)}{t}$ on $ {{}^t B}_{r_0}(x_0)$
for all $0\leq t \leq t_0$, for some small $r_0 = r_0(v_0) >0$ and some large $\ti c(v_0)$,  if
$\si_0(v_0)$ is chosen sufficiently small, and $x_0$ is an arbitrary
point in $\overline{{{}^0 B}_{ L -N(1
   -\be_0)}(p_0)}$. That is, the estimate of Step 1 for $0 \leq t  \leq  \min(1 -
10\be_0,t_0)$, can be extended to $0 \leq t \leq t_0$ (after changing $c_0(v_0)$ to a larger constant 
$\ti c(v_0)$) on a small time dependent
neighbourhood of $x_0$ if necessary: it is only
necessary to do this if $ t_0 > 1 - 10\be_0$.
Using these estimates, we then show that 
$|\Riem(\cdot,t)|\leq \frac{\ti c(v_0)}{t}$ on  the very large ball ${{}^t B}_{ \frac{
    \be_0 N}{64} }(y_0) $  for $ 0 \leq t \leq t_0$, for all $y_0 \in  \overline{{{}^0 B}_{ L -N(1-\frac{3}{4}\be_0)}(p_0)}  $, and hence, using Theorem 5.1 of \cite{SimSmoo} combined
with (a) and (d), we see that
$\curlR \geq -\frac{1 }{N}$ on  ${{}^t B}_{ \frac{ \sqrt{N}}{2} }(y_0) $  for all $ 0\leq t \leq t_0$ for all $y_0
\in \overline{{{}^0 B}_{ L -N(1-\frac{3}{4}\be_0)}(p_0)}$.
\hfill\break
{\bf Now we present the details of Step 2}.
\\
Let $\ep(3), \de(3)$ be the constants introduced in Step 1: $\ep(3),
\de(3)$ are the constants which appear in the second
Pseudolocality Theorem of G. Perelman, Theorem 10.3 in \cite{Per}, in
the case that $n=3$ (as it is here).
Assume $t_0>
1-10\be_0$. We know $t_0 \leq 1$.
Using  Theorem 10.3 of \cite{Per}, combined with the estimates
\eqref{localinfo2},  and \eqref{curvestshi2}, we  get 
\begin{eqnarray*}
|\Riem(x,t)| && \leq
\frac{1}{(\ep(3) R_0)^2} \mbox{ for all }  \ x \in  {}^t
B_{\ep(3) R_0} (x_0),  \cr
&& \  \ \ \  \ \mbox{ for all }  \ t \in [1-10\be_0,t_0) \cap [1-10\be_0 ,
1-10\be_0 +  R_0^2(v_0) \ep^2(3)).
\end{eqnarray*}
If we choose $\be_0 = \be_0(v_0) = \si_0^{1/4}(v_0)>0$ small enough, so that 
$R_0^2(v_0) \ep^2(3) >  10\be_0,$ then we have
$1-10\be_0 + R_0^2(v_0) \ep^2(3) >1$ und hence 
$[1-10\be_0,t_0) \cap [1-10\be_0 ,
1-10\be_0 +  R_0^2(v_0) \ep^2(3)) = [1-10\be_0,t_0)$, since $t_0\leq 1$.
This means that 
\begin{eqnarray*}
|\Riem(x,t)| \leq
\frac{1}{(\ep(3) R_0)^2} \mbox{ for all }  \ x \in  {}^t
B_{\ep(3) R_0} (x_0), \ \   t \in [1-10\be_0,t_0)
\end{eqnarray*}
Combining this
with \eqref{curvest} we see that
\begin{eqnarray}
|\Riem(x,t)| \leq \frac{\ti c(v_0)}{t}  
\mbox{ for all }  x \in {{}^t B}_{r_0}(x_0) ,\ \ 0\leq t \leq t_0 \label{curvest2}
\end{eqnarray}
for some small $r_0(v_0)= \ep(3) R_0(v_0) >0$ for all $x_0 \in \overline{{{}^0 B}_{ L -N + N\be_0}(p_0)}$, where $\ti c(v_0) = \max( \frac{1}{\ep^2(3) R_0^2(v_0)}, c_0(v_0)).$
That is, we have extended the estimates \eqref{curvest} up to time
$t_0$ on a {\it small} time dependent ball of fixed radius with middle point
$x_0$, for arbitrary $x_0 \in \overline{{{}^0 B}_{ L -N +
    N\be_0}(p_0)}$.\\
This is the point where we determine
$\be_0(v_0)=\si_0^{1/4}(v_0)$: it is now fixed for the rest of the
argument.
We stress the following point. The constants $\ti c(v_0)$ from \eqref{curvest2} and the constant
$\si_0(v_0)$, and hence $\be_0(v_0) = (\si_0(v_0))^{1/4}$) are now fixed. They  only depend on $v_0>0$. They do not
depend on $N$: we still have the freedom to choose $N$ as large as we
like without changing $\ti c(v_0),R_0(v_0),c_0(v_0), \be_0(v_0)$, or $\si_0(v_0)$. 
In fact decreasing $\si(v_0)$ (and hence $\be_0(v_0) =
(\si_0(v_0))^{1/4})$ and increasing $N$ would not change $\ti c(v_0),R_0(v_0)$
or $c_0(v_0)$ from above, in view of the definitions of $c_0(v_0),R_0(v_0)$ and $\ti c(v_0)$.

In order to get estimates on a {\it large} time dependent ball, we
restrict to points $y_0 $ in $\overline{{{}^0 B}_{ L
    -N(1-\frac{3}{4}\be_0)}(p_0)} \subseteq \overline{{{}^0 B}_{ L -N +
    N\be_0}(p_0)}$ and use the estimates that we have just obtained.
Let $y_0 $ in $\overline{{{}^0 B}_{ L
    -N(1-\frac{3}{4}\be_0)}(p_0)}$ be arbitrary.
Let $z \in \boundary( {{}^{0}
  B}_{\frac{\be_0N}{32}}(y_0)) $. Then, using the estimate
\eqref{curvest2}, we see that $ |\Riem(\cdot,t)| \leq \frac{\ti c(v_0)}{t}$ on ${{}^t B}_{r_0}(z)$ for
all $0 \leq t \leq t_0$ for some small fixed $r_0(v_0)>0$, and the same is true on ${{}^t
  B}_{r_0}(y_0)$, since $z,y_0 \in \overline{{{}^0 B}_{ L -N +
    N\be_0}(p_0)}$ due to the triangle inequality. Hence, using  Lemma \ref{Littlericci}, 
 we see that $d_t(y_0,z)
\geq d_0(y_0,z) -a_0(v_0) = \frac{\be_0N}{32} -a_0(v_0) >
\frac{\be_0N}{64}$ for all $t  \in [0,t_0]$, where  $a_0(v_0) = a(3,r_0(v_0), \ti c(v_0) )$ is the constant coming
 from Lemma \ref{Littlericci} (with $n=3$, $\ell=\ti c(v_0)$ and $j_0 = r_0(v_0)$ there), and we assume without loss of generality that $ \frac{N \be_0}{64} > a_0(v_0)$. Hence, since $z \in \boundary( {{}^{0}
  B}_{\frac{\be_0N}{32}}(y_0)) $ was arbitrary, it must be that
  $\overline{{{}^t B}_{\frac{\be_0 N}{64} }(y_0)} \subseteq  {{}^{0}
  B}_{\frac{\be_0N}{32}}(y_0)  $ remains true for all $0 \leq t \leq
t_0$: if there exists a first $t \in [0,t_0]$ were $\overline{{{}^t B}_{\frac{\be_0 N}{64} }(y_0)} \subseteq  {{}^{0}
  B}_{\frac{\be_0N}{32}}(y_0)  $ is violated, then there must  exist
a point $z \in \boundary( {{}^{0}
  B}_{\frac{\be_0N}{32}}(y_0)) \cap \overline{   {{}^t B}_{\frac{\be_0
      N}{64} }(y_0) }$ for this $t$, which contradicts  the inequality $d_t(y_0,z) >
  \frac{\be_0N}{64}$ that we just showed.\\
This implies that
\begin{eqnarray}
&& |\Riem(x,t)| \leq \frac{\ti c(v_0)}{t} \label{curvfinal1}
\end{eqnarray}
for all $x \in {{}^t B}_{\frac{\be_0 N}{64} }(y_0) $ for all $y_0 \in
\overline{{{}^0 B}_{ L -N(1-\frac{3}{4}\be_0)}(p_0)}$ in view of \eqref{curvest2}
and the fact that $ {{}^t B}_{\frac{\be_0 N}{64} }(y_0)
\subseteq \overline{ {{}^0 B}_{\frac{\be_0 N}{32} }(y_0)} \subseteq
\overline{{{}^0 B}_{ L -N + N\be_0}(p_0)} $ remains true for all $0 \leq t \leq t_0$.
Let $r =\sqrt{N}$. Then we have 
\begin{eqnarray*}
\curlR && \geq -\frac 1 {L^2} \cr
&& \geq -\frac 1 {N^2} \cr
&& = -\frac{ 1}{\ti c 400}( \frac {\ti c 400}{N^2}) \cr
&& \geq -\frac{ 1}{\ti c 400}( \frac{1}{N}) \cr
&& = -\frac{ 1}{400 \ti c r^2}
\end{eqnarray*}
at time zero on ${{}^{0}B}_r(y_0)$ since, without loss of generality, $\ti
c(v_0) 400 \leq N$ and $\sqrt{N} \leq \frac{\be_0 N}{64}$ which tells us that ${{}^{0}B}_r(y_0) =  {{}^{0}B}_
{\sqrt{N}}(y_0) \subseteq {{}^0 B}_{\frac{\be_0 N}{64} }(y_0)\subseteq
\overline{{{}^0 B}_{ L -N + N\be_0}(p_0)} \subseteq {{}^0B}_L(p_0)$.
Now using Theorem 5.1  of the  paper \cite{SimSmoo} , we see that
\begin{eqnarray}
\curlR(x,t) \geq - \frac{1}{r^2} = -\frac{1 }{ N} \label{contra}
\end{eqnarray}
for all $x \in {{}^t B}_{\frac r 2} (y_0) = 
{{}^t B}_{\frac{\sqrt{N}
}{2}   }(y_0),$ for all $y_0 \in
\overline{{{}^0 B}_{ L -N(1-\frac{3}{4}\be_0)}(p_0)},$
for all $0\leq t \leq \min(\hat \de^2(v_0) r^2,t_0)
=t_0$, where $\hat \de(v_0)= \de(v_0,\ti c(v_0), \frac{1}{2}) $ is the constant coming from Theorem 5.1 of \cite{SimSmoo},
since without loss of generality $\hat \de^2(v_0) r^2 = \hat \de^2(v_0) N \geq
1$, and
$|\Riem(\cdot,t)| \leq \frac{\ti c(v_0)}{t}$
on  ${{}^t B}_{\sqrt{N}  }(y_0)$ in view of equation \eqref{curvfinal1},
since  ${{}^t B}_{\sqrt{N}  }(y_0) \subseteq
{{}^t B}_{\frac{\be_0 N}{64} }(y_0),$ where here we are using again the fact
that $\frac{\be_0 N}{32} \geq \sqrt{N}$.  To apply Theorem 5.1 of
\cite{SimSmoo} here, scale so that $r=1$ and then scale the conclusion
of the Theorem 5.1 of \cite{SimSmoo} back to the case $r=\sqrt{N}$, to
obtain the estimate claimed here (the $N$ appearing in Theorem 5.1 of \cite{SimSmoo}  is $N :=\ti c$, where $\ti
c$ is the $\ti c$ appearing in the current proof: the  $N$ of the  theorem we are
proving has nothing to with the $N$ of Theorem 5.1 of \cite{SimSmoo}) .                                    
\\
{\bf This finishes Step 2}
\\
{\bf Step 3}\\
In Step 3 we use the estimates from above to show that
the contradiction point $z_0$ from the beginning of this argument 
must in fact be in $ \overline{{{}^0 B}_{ L -N(1-\frac{3}{4}\be_0)}(p_0)}$. This along with the fact
that $\curlR(z_0)(t_0)  \geq -\frac{1 }{N}$ for such points (Step 2) and
$(c)$ leads to a contradiction if $N=N(v_0)>0$ is large enough.\\
{\bf Now we present the details of Step 3}.
\\
Consider once again elements $y_0 \in \overline{{{}^0 B}_{ L -N(1-\frac{3}{4}\be_0)}(p_0)} \subseteq \overline{{{}^0 B}_{ L -N + N\be_0}(p_0)}$, where 
 $\be_0= \be_0(v_0)$  is as defined in the Steps 1,2 above.

The estimate  \eqref{curvfinal1} above combined with Lemma
\ref{Littlericci} shows that for \\$y$ in $\boundary( {{}^0 B}_{ L -N(1-\frac{3}{4}\be_0)}(p_0))$, that
is for $y$ with $d_0(y,p_0) =  L -N + \frac{3}{4} \be_0N$,
we have 
\begin{eqnarray*}
d_t(y,p_0) \geq d_0(y,p_0)  -a_1 (v_0) && = L-N  + \frac {3}{4} \be_0 N
-a_1(v_0) \cr
&& \geq L-N  + \frac {N\be_0}{ 2 } 
\end{eqnarray*} for all
$t \leq t_0$, since without loss of generality, 
$\frac{\be_0 N}{4} \geq a_1(v_0)+1$, where we used that $p_0$
is also contained in $\overline{{{}^0 B}_{ L -N(1-\frac{3}{4}\be_0)}(p_0)}$, and 
$a_1(v_0) = a(3,1, \ti c(v_0))$ is the constant from Lemma
\ref{Littlericci} (with $n=3$, $\ell=\ti c(v_0)$ and $j_0 = 1$ there). This implies that $$
d_t(y,p_0) \geq L-N  + \frac{1}{2}\be_0 N$$ for all 
$y \in (\ \overline{{{}^0 B}_{ L -N(1-\frac{3}{4}\be_0)}(p_0)} \ )^c,$ for all
$0\leq t \leq t_0$ (every length minimising geodesic with respect to $g(t)$ which joins $p_0$ to
$y$, where $y$ is outside of $\overline{{{}^0 B}_{ L
    -N(1-\frac{3}{4}\be_0)}(p_0)}$, must intersect $\boundary( {{}^0 B}_{ L -N(1-\frac{3}{4}\be_0)}(p_0))$), and
hence $L- d_t(y,p_0) \leq N-
\frac{1}{2}\be_0 N$, which means $\dist_t(y)  = \max(0,L- d_t(y,p_0))  \leq \max(0,N-
\frac{1}{2}\be_0 N) = N-
\frac{1}{2}\be_0 N$, for all
$t \leq t_0$  for such points $y \in (\ \overline{{{}^0 B}_{ L
    -N(1-\frac{3}{4}\be_0)}(p_0)}\ )^c$.
In particular $z_0$ is not in $(\ \overline{{{}^0 B}_{ L
    -N(1-\frac{3}{4}\be_0)}(p_0)}\ )^c$  since $ \dist_{t_0}(z_0) =
N$ (we scaled so that this is true),
and hence $z_0 \in \overline{{{}^0 B}_{ L -N(1-\frac{3}{4}\be_0)}(p_0)}$.
Now using the fact (inequality \eqref{contra}) that 
$\curlR(\cdot,t) \geq-\frac{1}{  N }$ on  ${}^t B_{\frac{\sqrt{N}}{2}
}(z_0) $ for all $t \in [0,t_0]$ we obtain a contradiction to (c), 
$\curlR(z_0,t_0)( \omega , \omega) = -\si_0$,  if $N$ is chosen large
enough, for example $N \geq \frac{2}{\si_0}$.
\\
{\bf This finishes Step 3 and the proof of the Theorem}.
\end{proof}
\begin{remark}\label{constants}

The following important constants appeared, in this order, in the
proof of the theorem:
$c_0(v_0) = c_0(3,v_0,1/2,16), m_0(v_0) = m_0(3,v_0,1/2,16)$, where
$ c_0(3,v_0,1/2,16), m_0(3,v_0,1/2,16)$ are the
constants coming from Theorem \ref{thm22new},\\
$R_0(v_0)= R_0(3,2c_0(v_0),\ti v_0(v_0),\de(3))$, where $R_0(3,
2c_0(v_0),\ti v_0(v_0),\de(3))$ is the constant coming from
Theorem \ref{littletechnical} and $\de(3)$ and $\ep(3)$ are the
constants coming from G. Perelman's Pseudolocality Theorem 10.3 in
\cite{Per}), $r_0(v_0) = \ep(3)R_0(v_0)$, \\ 
$\ti c(v_0) = \max( \frac{1}{\ep^2(3)R_0^2(v_0)}, c_0(v_0))$,
$a_0(v_0) = a(3,r_0(v_0), \ti c(v_0))$, 
where $a(\cdot,\cdot,\cdot)$ is the constant coming from Lemma
\ref{Littlericci}, $\hat \de(v_0) := \de(v_0, \ti c(v_0),\frac{1}{2})$, where
$\de(\cdot,\cdot,\cdot)$ is the constant coming from Theorem 5.1 of \cite{SimSmoo},
$a_1(v_0) = a(3,1,\ti c(v_0))$,
where $a(\cdot,\cdot,\cdot)$ is the constant coming from Lemma
\ref{Littlericci}.  
The following assumptions on the largeness  of $N(v_0)$ and smallness of
$\si_0(v_0)$ and $\be_0(v_0)$ were used in the proof: 
$\be_0^2 \leq m_0(v_0),$ $\frac{N \be_0}{2} \geq \frac{2}{
   \be_0},$  $\frac{\be_0N}{4} \geq a_0(v_0),$  $1-10\be_0 + \ep(3)
 R_0^2(v_0)>0.$ This fixes the constants $\be_0$ and $\si_0$. We
 further require $N \geq a_0(v_0) \frac{64}{\be_0},$ $N \geq 400 \ti
 c(v_0)$, $\sqrt{N}\be_0 \geq 64$, $N \hat \de^2(v_0) \geq 1,$ $\be_0
 N \geq 4 a_1(v_0)+1,$ $N \geq \frac{2}{\si}.$
This determines $N$. 
Hence it is possible to determine $\si$ (and hence $\beta_0 =
\si_0^{1/4}$) and $N$ in the first line of
the above given proof.
\end{remark}

\section{Proof of Theorem \ref{threed}}\label{secthreed}

{\it Proof of Theorem \ref{threed}}.
Scale so that $r=1$. 
Then we have due to the Bishop-Gromov volume comparison theorem,
\begin{itemize}
\item[(i)] $\vol({{}^{0} B_{r}(x)}) \geq v(\al,v_0)r^3$ for all ${{}^0 B}_r(x) \subseteq {{}^0 B}_{1-\al}(x_0)$,
\item $\curlR(g(0)) \geq -  \frac{1}{(1-\al)^2}$ on  ${{}^{0} B}_{1-\al}(x_0)$.
\end{itemize} 
See the Appendix in Version 1 of 'Local Smoothing Results for the Ricci flow in
dimensions two and three', M. Simon, arXiv:1209.4274v1 for a
reference: since the points $x$ are not at the centre of the ball
${{}^0 B}_{1-\al}(x_0)$, $v(\al,v_0)$ can depend on $\al$.
Hence, Theorem \ref{threed2} is valid for $r=1-\al$, and we find that
there exists $K = K(v(\al,v_0)) = K(\al,v_0) $ and $\si_0 =
\si_0(v(\al,v_0)) = \si_0(\al,v_0) > 0$ , 
such that
$$\curlR(g(t))(x) (1-\al-d_t(x,x_0) )^2 > -  K^2  $$
for all  $x \in {}^{t} B_{1-\al}(x_0) $  which satisfy 
$ (1-\al-d_t(x,x_0))^2 \geq \frac{K^2}{\si_0}t$ and $ t \leq 
\frac{ \si_0(1-\al)^2}{K^2} $ and $t \in [0,T)$.
In particular, $\curlR(g(t))(x) > - \frac{ K^2}{\al^2} (*) $
for all  $x \in {}^{t} B_{1-2\al}(x_0) $ for all  $ t \leq 
\min( \frac{\si_0 \al^2 }{K^2}, \frac{ \si_0 (1-\al)^2}{ K^2})$ with $t \in [0,T)$.
Now we may use Theorem \ref{thm22new}, with $r= 1-2\al$ to further conclude that,
$|\Riem(x,t)| \leq \frac{c_0(\al,v_0)}{t} (**) $ for all $x \in {}^{t} B_{1-4\al}(x_0) ,$
$t \leq S(\al,v_0)$. Choosing
$\al = 1/10$ in the above argument, we see that 
we also get  $\vol( {{}^{t}B}_{1}(x_0 ) )  \geq \ti v_0(v_0)
(4/5)^3 (***)$ for all $t \leq S(\al,v_0)$, for some $\ti v_0 (v_0) >0$.  The estimates (*),(**) and (***) are the desired estimates.\hfill
$\Box$

\begin{appendix}
\section{Dimension of Gromov-Hausdorff limits of collapsing and
  non-collapsing spaces}
\end{appendix}
We explain why some certain well known properties of collapsing, respectively
non-collapsing manifolds, with curvature bounded from below
hold.  These properties follow from
the results contained in \cite{BGP} (see also \cite{BBI}).
Note that the definition of {\it Alexandrov space with curvature bounded
from below}  in \cite{BGP} (Definition 2.3)  and \cite{BBI} (Proposition 10.1.1) agree.
\begin{theorem}\label{appendixthm}
Let $(B_1(p_i),g_i)$, $(B_1(q_i),h_i)$, $i \in \N$  be balls whose
closure is compactly
contained in smooth
Riemannian manifolds without boundary of dimension $n \in\N$ fixed. Assume that
$\sec \geq -V$ on these balls and that \\
$d_{GH}( (B_1(p_i),g_i), (B_1(q_i),h_i) ) \to 0$ as  $i \to \infty$,
and  $\vol( (B_1(p_i),g_i)) \geq v_0>0$  for all $i
\in \N$.
Then  it cannot be, that
 $\vol(B_1(q_i),h_i) \to 0$ as $i
\to \infty$.

\end{theorem}
\begin{proof}
Assume the theorem is false.
We know that $(B_{1}(p_i),g_i,p_i)$ and $(B_1(q_i),h_i,q_i)$ Gromov-Hausdorff converge,
after taking a subsequence, to the same space $(X=B_1(p),d,p)$ by the theorem
of M. Gromov, and that $(X,d,p)$ is an Alexandrov space (see {\it
  Notes on Alexandrov Spaces} below).
Without loss of generality, we may assume that $\sec \geq -k^2$ on the
balls we are considering, where
$k^2>0$ is as small as we like. This can be seen as follows.
Without loss of generality (renumber the indices $i$), we have 
 $\vol(B_{1/i}(q_i),h_i) \leq \vol(B_1(q_i),h_i) \leq \frac{1}{i^{n+1}}.$
The Bishop-Gromov Comparison principle implies that 
$\vol(B_{1/i}(p_i),g_i) \geq c(v_0,n)\frac {1}{i^n}.$
Scaling both Riemannian metrics by $i^2$, we have (we also call the rescaled
metrics $g_i$ and $h_i$)
$\vol(B_{1}(p_i),g_i) \geq c(v_0,n)>0$
and $\vol(B_{1}(q_i),h_i)  \leq \frac{1}{i}$ and $\sec \geq - \frac{V}{i^2}.$
So we assume $\sec \geq -k^2$ with $k > 0$ arbitrarily small.

Let $\overline{B_{a}(y)} \subseteq  B_{1-10a}(p)$ and
let 
$\{B_{R/3}(s_j) \}_{j\in \{1, \ldots,N \} }$  be any maximally  pairwise
disjoint collection of balls with  $R<<a<1/(10)$ and
centres $s_j$ in $B_{a}(y)$.
By {\it maximally pairwise disjoint}  we mean, that if we try and add a ball $B_{R/3}(z)$
to the collection, where $z \in B_{a}(y)$, then the new collection
is not pairwise disjoint. Then clearly 
$\{ B_{R}(s_j) \}_{j\in \{1, \ldots N\} }$  must cover $B_{a}(y)$.
Let $\ti s_j$ respectively $\ti p = p_i$, $\ti y$ be the corresponding points in
$ (B_1(p_i),g_i,p_i)$ which one obtains by mapping $s_j$ respectively $p$,$y$ back
to $ (B_1(p_i),g_i,p_i)$ using the Gromov-Hausdorff approximation 
$f_i:(B_1(p),d,p) \to  (B_1(p_i),g_i ,p_i)$: we write
$p_i = \ti p$, and so on, suppressing the dependence of the points on
$i$ sometimes, in order to make this explanation more readable. For $i $ large enough, 
$\{ B_{2R}(\ti s_j) \}_{j\in \{1, \ldots N\} }$ must cover $B_{a}(\ti y)$ and
$\{B_{R/4}(\ti
s_j)\}_{j\in \{ 1,\ldots ,N\}}$  must be pairwise disjoint and
contained in  $B_{2a}(\ti y) \subseteq B_1(\ti p)$.
The Bishop-Gromov volume comparison principle implies that
$\frac{c_2(v_0,a,n)}{R^n}   \leq N \leq \frac{c_1(v_0,a,n)}{R^n},$ for
some fixed $0<c_0(v_0,V,a,n),c_1(v_0,a,V,n) < \infty$ 
and hence the rough dimension of $B_{a}(y)$ (see Definition 6.2
in\cite{BGP}) must be
$n$.

This means that the Hausdorff-dimension and burst index of $B_s(p)$ is
also $n$ for all $s<1$ (see Lemma 6.4 and Definition 6.1 in \cite{BGP}).
Assume $\ep \leq \frac{1}{1000n}$  in all that follows.
Now let $z \in B_{1/4}(p)$ be a point for which there is an $(n,\ep)$
explosion (Definition 5.2 in
\cite{BGP}: an $(n,\ep)$ explosion is called an $(n,\ep)$
strainer in \cite{BBI}, see Definition 10.8.9 there).
Note that for any $\frac{1}{1000n} \geq \ep >0$ such
a point exists (see Corollary 6.7 in \cite{BGP}).
Let $(a_k,b_k)_{k \in \{1,\ldots n\}} $ be such an $(n,\ep)$ explosion
at $z$ and assume that $a_k,b_k \in B_{s}(z)$ for all $k =1,
\ldots,n$ with $s<<1$: as pointed out in \cite{BGP} (just after Definition 5.2), we can
always make this assumption, see also Proposition 10.8.12 in
\cite{BBI}.
Then there exists a small ball $B_r(z)$ such that $(a_k,b_k)_{k \in \{1,\ldots n\}}$ 
is an $(n,\ep)$ explosion at $x$ for all $x \in B_{r}(z)$ and
$(a_k,b_k)_{k \in \{1,\ldots n\}}$ 
is in $B_{s}(p) \backslash B_{2 \hat r}(z)$ where $s >>\hat r >> r>0$: distance is
continuous in $X$ and comparison angles (which are measured in $M^2(-V):=$
  hyperbolic space with curvature equal $-V$) 
change continuously as distances change continuously and stay away from
zero (see \cite{Mey}, equation (44)).
With $s >>\hat r >> r$, we mean $\frac{\hat r}{s} << 1$ and $ \frac{
  r}{\hat r} << 1.$  
Going back to $(B_1(q_i),h_i,q_i)$ with our Gromov-Hausdorff approximation, we see
(once again dropping dependence on $i$ for readability)
that  there exists a ball $B_r(\ti z) \subseteq B_{1/2}(q_i)$ and an
explosion $(\ti a_k,\ti b_k)_{k \in
  \{1,\ldots n\}} $ in  $ B_{2s}(\ti z) \backslash B_{\hat
  r}(\ti z)$  (if $i$ is large enough)
such that $(\ti a_k,\ti b_k)_{k \in
  \{1,\ldots n\}} $ is an $(n,4\ep)$  explosion at $x$ for all $x  \in B_r(\ti z)$: once again,
  this follows from the fact that angle comparisons change continuously
  as distances change continuously and stay away from zero, and
  distance changes at most by $\de(i)$, with $\de(i) \to 0$ as $i \to
  \infty$, under our Gromov-Hausdorff
  approximation.
There are no $((n+1),\ep)$ explosions in $(B_{1}(q_i),h_i,q_i)$, as the Hausdorff dimension of the
manifold (and hence the burst index) is $n$ (see Theorem 5.4 in
\cite{BGP} or  Proposition 10.8.15  in \cite{BBI}). Fix  $0<\ep(n) <<  \frac{1}{2000n}$.
But then, using Theorem 5.4 in
\cite{BGP}, see also 
 Theorem 10.8.18 in \cite{BBI}, (more explicitly, using  the
proofs thereof) we see that there is a $\ti r =\ti
r(n,r)>0$ and  a bi-Lipschitz homeomorphism from
$ f: B_{\ti r}( \ti z ) \to f( B_{\ti r}(\ti z)) \subseteq \R^n,$ where the bi-Lipschitz constant
may be estimated by $ \frac{1}{c(n)}d_i(x,y) \leq |f(x) - f(y)|  \leq
c(n) d_i(x,y)$ for some $c(n) >0$,
 and hence $\vol( B_1(q_i),h_i,q_i) \geq \ep(n,r) >0$  for $i$ large
 enough, as $r,n$ do not depend
on $i$.
This shows, that after taking a subsequence, we must have
$\vol(B_1(q_i),h_i,q_i) \geq \ep(n,r)>0$.
\end{proof}

{\bf Notes on  Alexandrov Spaces}\hfill\break

The fact that
$(B_{1}(p_i),g_i,p_i)$ and $(B_1(q_i),h_i,q_i)$ Gromov-Hausdorff converge to some metric space $(X = B_1(p),d)$ after
taking a subsequence follows from Gromov's Convergence Theorem (we apply
the theorem to the closed balls $\overline{B_{1-\frac 1 i}(p)}
\subseteq B_1(p)$ with $i \in \N$, and then take a diagonal subsequence). See
10.7.2 in \cite{BBI}. 
The limit space has the property that $\overline{B_s(p)}$ is complete
for all $0<s<1$ (by construction), and $\overline{B_s(p)}$ is compact for all $0<s<1$,
since it is also totally bounded (due to the Bishop-Gromov comparison
principle: see the argument on  the rough dimension  of
$\overline{B_a(y)}$ at the beginning of the proof above).

In order to guarantee that 
  $(X=B_1(p),d,p)$  is an {\it Alexandrov space}, a local version of the {\it Globalisation Theorem of
  Alexandrov-Toponogov-Burago-Gromov-Perelman} (Theorem 3.2 in
\cite{BGP}) is necessary, as the spaces we are considering are not
complete.
Such a local version of the theorem exists, as pointed out in Remark 3.5 in \cite{BGP}.
Proofs of the Globalisation Theorem can be found in the book \cite{AKP}  and a
similar proof, obtained independently,  is given in the paper \cite{LS}.
Examining the proofs of the Globalisation Theorem (in the case  $\sec
\geq -1$) in any of the proofs mentioned above, we see that the
proofs are local. Examining any of the proofs mentioned above, we see  that the following is
true: if $ (B_1(x_0),g)$ is compactly 
contained  in a smooth manifold, and  $\sec
\geq -1$ on $ (B_1(x_0),g)$ and $ z \in B_1(x_0)$ has $d(x_0,z) = 1-r$, then the quadruple condition (or the
hinge condition , or any of the other equivalent conditions, see 
section 2 in \cite{BGP} or 8.2.1 in \cite{AKP}, or the discussion on
page 3 of \cite{LS} to see why these conditions are equivalent)  hold
on the ball $B_{rc}(z) \subseteq  B_1(x_0)$ for some fixed constant
$0< c<< 1$ independent of $z$ or $r$. Note that the space $(X=B_1(p),d)$ we obtain this way is {\it locally
  intrinsic}: for all $x \in X$, for all  $z,q \in B_{\ep}(x)$ for
  all $ B_{5\ep}(x) \subseteq B_{1-\al}(p)$ for all $1>\al,\ep>0$ there exists a length
  minimising geodesic between $z$ and $q$ which is contained in
  $B_{5\ep}(x)$: see the proof of Theorem 2.4.16 in \cite{BBI}.



\end{document}